\documentclass{amsart}

\usepackage{amsmath,amssymb,amsthm}
\usepackage{graphicx,tikz}

\newtheorem*{thm}{Theorem}

\newtheorem*{corollary}{Corollary}

\newtheorem{lemma}{Lemma}

\theoremstyle{definition}

\theoremstyle{remark}

\DeclareMathOperator{\q}{q-quant}

\begin{document}

\title[]{Quantile-Based Random Kaczmarz for\\ corrupted linear systems of equations}
\thanks{S.S. is supported by the NSF (DMS-2123224) and the Alfred P. Sloan Foundation.}

\author[]{Stefan Steinerberger}
\address{Department of Mathematics, University of Washington, Seattle}
\email{steinerb@uw.edu}

\begin{abstract} We consider linear systems $Ax = b$ where $A \in \mathbb{R}^{m \times n}$ consists of normalized rows, $\|a_i\|_{\ell^2} = 1$, and where up to $\beta m$ entries of $b$ have been corrupted (possibly by arbitrarily large numbers). Haddock, Needell, Rebrova \& Swartworth propose a quantile-based Random Kaczmarz method and show that for certain random matrices $A$ it converges with high likelihood to the true solution. We prove a deterministic version by constructing, for any matrix $A$, a number $\beta_A$ such that there is convergence for all perturbations with $\beta < \beta_A$. Assuming a random matrix heuristic, this proves convergence for tall Gaussian matrices with up to $\sim 0.5\%$ corruption (a number that can likely be improved).
\end{abstract}

\maketitle

\section{Introduction}

\subsection{The Kaczmarz method} We start by explaining the Kaczmarz method \cite{kac} for uncorrupted linear systems of equations.
Let $A \in \mathbb{R}^{m \times n}$, $m \geq n$, and suppose that
$ Ax =b,$
where $x \in \mathbb{R}^n$ is the (unknown) solution of interest and $b$ is a given right-hand side.
Throughout this paper, we use $a_1, \dots, a_m$ to denote the rows of $A$. 
 One way
of interpreting this system geometrically is to write it as 
$$ \left\langle a_i, x \right\rangle = b_i \qquad \mbox{for}~1 \leq i \leq m$$
and to interpret the solution as the intersection of hyperplanes. The idea behind the Kaczmarz method is as follows: given an approximation
of the solution $x_k$, take an arbitrary equation, say the $i-$th equation, and project $x_k$ onto the hyperplane described by
the $i-$th equation $ \left\langle a_i, x \right\rangle = b_i$, formally:
$$ x_{k+1} = x_k + \frac{b_i - \left\langle a_i, x_k\right\rangle}{\|a_i\|^2}a_i.$$

 The Pythagorean theorem implies that $\|x_{k+1} - x \| \leq \|x_k - x\|$ suggesting convergence of the method. It is not easy to make this quantitative \cite{gal}.
 Strohmer \& Vershynin \cite{strohmer} proposed to randomize the method and determined the convergence rate in terms of the smallest singular value and the Frobenius norm of $A$.

\begin{thm}[Strohmer \& Vershynin,  \cite{strohmer}] If $a_i$ is chosen with likelihood $\|a_i\|^2$, then
$$ \mathbb{E}~ \|x_k - x\|^2 \leq \left(1 - \frac{1}{\|A^{-1}\|^2 \cdot \|A\|_F^2}\right)^k \|x_0 - x\|^2.$$
\end{thm}

This rate is known to be essentially best possible \cite{stein}.  The work of Strohmer \& Vershynin has inspired a lot of subsequent work, see \cite{eldar, elf, gower, gower2, jiao, leventhal, liu, ma, moor, need, need2, need25, need3, need4, nutini, popa, stein2, stein3, stein4, strohmer, zhang, zouz}.

\subsection{Corrupted Linear Systems.} 
Let us now suppose that we are interested in finding the solution $x$ of the linear system $Ax = b_t$, where $b_t \in \mathbb{R}^m$ is the true right-hand side. However,
instead of the true right-hand side $b_t$, we are only given
$$  b = b_t + \mbox{error} \qquad \mbox{where}~\mbox{error} \in \mathbb{R}^m$$
and we know that the error is supported on a $\beta-$fraction of its length 
$$ \| \mbox{error} \|_{\ell^0} \leq \beta m.$$
We will not make any further assumptions on the error (in particular, the entries could be arbitrarily large). Is it then still possible to recover the true solution $x$?
We assume, throughout the paper, that all rows of $A$ are normalized: $\|a_i\|_{\ell^2} = 1$. 
A fascinating approach was recently proposed by Haddock, Needell, Rebrova \& Swartworth \cite{hadd}: given an approximate solution $x_k$, consider the set $ \left\{ \left| \left\langle x_k, a_i \right\rangle - b_i \right| : 1 \leq i \leq m \right\}$. This set measures, essentially, how `wrong' each of the equations is. The Random Kaczmarz method would now pick one of the equations uniformly at random. Haddock, Needell, Rebrova \& Swartworth \cite{hadd} propose to instead look at the $q-$th quantile of the set, these are the $qm$ `least incorrect' equations and then pick one at random and use that for a step of the Kaczmarz method (see also \S 4.1). \\

By looking at equations that are only violated `a little', we are, hopefully, more likely to consider equations that are not corrupted (because corrupted equations are presumably violated by a lot); moreover, even if we were to pick an incorrect equation (one that was actually corrupted),
selecting one with little overall error ensures that the error incurred in this step is not too large: the update may remove us from the correct solution but not by too much. Haddock, Needell, Rebrova \& Swartworth \cite{hadd} show that if $A$ is a random
matrix of a certain type, the method can recover the true solution with high likelihood as long as the support of the corruption $\beta$ is sufficiently small.

\begin{thm}[Haddock, Needell, Rebrova \& Swartworth, \cite{hadd}] For a certain class of random matrices $A \in \mathbb{R}^{m \times n}$ and $m/n$ sufficiently large, if
the support of the error is sufficiently small, $\beta < \min(c_1 q^2, 1-q)$, the $q-$quantile Random Kaczmarz method converges exponentially with likelihood at least $1 - c_2 \exp(-c_q m)$.
\end{thm}

The paper \cite{hadd} also demonstrates, empirically, that the method works remarkably well and can handle both substantial amounts of error and real-life data.   

\subsection{Related Results.} The $q-$quantile method is vaguely related to an earlier approach for the uncorrupted problem: given 
$\left\{ \left| \left\langle x_k, a_i \right\rangle - b_i \right| : 1 \leq i \leq m \right\},$ one could wonder whether it would not make sense to project onto the hyperplane corresponding to the equation that is violated the `most'. This is sometimes known as Motzkin's method \cite{agmon, motzkin} or the maximal correction method \cite{cenker}. This works well in practice, one observes a faster rate of convergence, and this has been investigated by  Bai \& Wu \cite{bai0, bai, bai2},  Du \& Gao \cite{du}, Gower, Molitor, Moorman and Needell \cite{gower2}, Haddock \& Ma \cite{haddock}, Haddock \& Needell \cite{haddock2}, Jiang, Wu \& Jiang \cite{jiang}, Li, Lu \& Wang \cite{lilu}, Li \& Zhang \cite{li}, Nutini, Sepehry, Laradji, Schmidt, Koepke \& Virani \cite{nutini} and the author \cite{stein2}. 
The approach proposed in \cite{hadd} is related to work of Haddock \& Needell \cite{haddock15, haddock16} where the Random Kaczmarz method was used to detect corruptions (see also \cite{hadd2}). We also refer to \cite{amaldi, jamil}.

\section{The Theorem}
\subsection{Setup.} We start with a formal description of the problem.

\begin{quote}
\textbf{Problem.}
\begin{enumerate}
\item Let $A \in \mathbb{R}^{m \times n}$ where all $m$ rows are normalized to $\|a_i\|_{\ell^2} = 1$. We want to find the solution  $x \in \mathbb{R}^n$ of $Ax = b_t \in \mathbb{R}^m$. We are only given a perturbation $b \in \mathbb{R}^m$ of $b_t$ satisfying
$$ \| b_t - b \|_{\ell^0} \leq \beta m.$$
\item Given $A$ and $b$, we want to reconstruct $x$.
\end{enumerate}
\end{quote}
Clearly, in order for a reconstruction to be possible, one will have to make some assumptions on $A$ and $\beta$. Our interest in the problem is inspired
by the $q-$quantile Random Kaczmarz method which was proposed in \cite{hadd} as a way of solving the problem. 
We first state the algorithm in an explicit form (our presentation differs slightly from the one in \cite{hadd}, these differences are immaterial and discussed in \S 4.1).
 \begin{quote}
 \textbf{Algorithm.} Given $A \in \mathbb{R}^{m \times n}$ with normalized rows, $\|a_i\|_{\ell^2} = 1$ for all $1 \leq i \leq m$, given $0 < q < 1$ and an approximate solution $x_k$:
 \begin{enumerate}
 \item Compute the $m$ numbers 
 $$ N_1 = \left\{1 \leq i \leq m: \left| \left\langle x_k, a_i \right\rangle - b_i \right| \right\}.$$
 \item Compute the $q-$quantile $Q$ of $N_1$ and consider all the equations that lie in the $q-$quantile
 $$ N_2 = \left\{ 1 \leq i \leq m: \left| \left\langle x_k, a_i \right\rangle - b_i \right| \leq Q \right\}.$$
\item  Choose an $i \in N_2$ uniformly at random and set
$$ x_{k+1} = x_k - (b_i - \left\langle a_i, x_k\right\rangle)a_i.$$
 \end{enumerate}
 \end{quote}
  
  Our goal is to provide explicit conditions on $A$ and $\beta$ under which this algorithm converges. It is clear that we have to somehow measure the `quality' of a matrix to be able to make quantitative statements: if all rows $a_i \in \mathbb{R}^n$ point in
different directions and if there are many such directions, then one could hope that one can compensate for some amount of corruption in the system. A natural example is that of random matrices with rows $a_i$ sampled uniformly at random from $\mathbb{S}^{n-1}$. A deterministic example would be given by spherical designs on $\mathbb{S}^{n-1}$ or unit norm frames with good condition number.
Conversely, if a subset of the rows $a_i$ of the matrix point in somewhat similar directions, they
capture similar aspects of the solution and a targeted corruption may prove to be more damaging. We will measure this quality using the parameter $\sigma_{q-\beta,\min}(A)$ (a version of which already appeared in \cite{hadd}) defined as
$$ \sigma_{q-\beta,\min}(A) = \min_{S \subset \left\{1,2,\dots, m\right\} \atop |S| = (q-\beta)m } \inf_{x \neq 0} \frac{\|A_S x\|}{\|x\|},$$
where $A_S$ is the matrix $A$ restricted to rows indexed by the subset $S$. This quantity measures whether restricting $A$ to a $(q-\beta)-$fraction of its rows
 can lead to a matrix with small singular values (see \S 2.3 for further comments on this quantity). Such a hypothetical
sub-matrix would correspond to a subsets of rows that capture very similar amounts of information which makes a matrix vulnerable to
corruption. We will also use the usual largest singular value of a matrix $\sigma_{\max}(A)$.

\subsection{The Result.} We can now state the main result.
\begin{thm}[Main Result] Assuming the setup described \S 2.1 and $\beta < q < 1-\beta$ arbitrary, if 
$$  \frac{q}{q-\beta}\left( \frac{2\sqrt{\beta} }{\sqrt{1 - q - \beta}} + \frac{\beta}{1 - q - \beta} \right) <  \frac{\sigma_{q-\beta,\min}^2}{ \sigma_{\max}^2}$$
then there exists $c_{A,\beta,q} > 0$ such that the $q-$quantile Random Kaczmarz method converges for \underline{\emph{all}} $\beta-$corruptions of the
linear system and
$$ \mathbb{E} \| x_{k} - x \|^2 \leq (1-c_{A, \beta, q})^k  \cdot \| x_{0} - x \|^2,$$
where
$$c_{A, \beta, q}  =(q-\beta)\frac{ \sigma_{q-\beta,\min}(A)^2}{q^2 m} -\frac{  \sigma_{\max}(A)^2}{q m} \left(\frac{2\sqrt{\beta} }{\sqrt{1 - q - \beta}} + \frac{\beta}{1 - q - \beta} \right) > 0.$$
\end{thm}
\vspace{7pt}
\textbf{Remarks.} 
\begin{enumerate}
\item
This is presumably not the optimal condition and not the optimal constant and it would be interesting to have sharper results (see also \S 2.3).
\item The algebraic structure of the condition requires $\beta < 1-q$. The condition $\sigma_{q-\beta,\min}(A) > 0$ requires
$(q-\beta)m \geq n$ and thus $\beta < q$ as well as $m > n$: the result (unsurprisingly) only applies to overdetermined systems.
\item One could slightly relax the condition on $ \sigma_{q-\beta,\min}(A)$: it would suffice to look at submatrices indexed by
rows corresponding to equations that have \textit{not} been corrupted, this is a smaller set and thus leads to a larger value for
this modified smallest singular value.
\item One could consider analogous methods for matrices without the normalization $\|a_i\|=1$, we refer to \S 4.3. More generally, there are a number of variations on the method that one could consider (different selection probabilities, for example, see \S 4.4).
\item Our result shows convergence under \textit{all} $\beta-$corruptions, i.e. all corruptions with $\| b-b_t\|_{\ell^0} \leq \beta m$. For
practical applications, it may be interesting to restrict to random $\beta-$corruptions. Naturally, one would expect stronger results in such a relaxed setting.
\end{enumerate}

\subsection{Regarding $\sigma_{q-\beta,\min}(A)$.} The strength of the result hinges on
$$ \sigma_{q-\beta,\min}(A) = \min_{S \subset \left\{1,2,\dots, m\right\} \atop |S| = (q-\beta)m } \inf_{x \neq 0} \frac{\|A_S x\|}{\|x\|}.$$
If $ \sigma_{q-\beta,\min}(A)$ is too small, then it will not allow for any nontrivial result since $\beta < 1/m$ would mean that not a single
equation can be corrupted. 
Computing $ \sigma_{q-\beta,\min}(A) $ for an explicitly given matrix $A$ might be somewhat difficult, however, at least for certain types of random matrices one can hope to get a decent understanding.
For many random matrices with $\|a_i\|=1$ one would expect
$$ \mathbb{E} ~\sigma_{\max}(A) \sim \sqrt{ \frac{m}{n}}$$
and in combination with the trivial inequality $\sigma_{q-\beta,\min}(A) \leq \sigma_{\max}(A)$ one can get a first idea of how things should scale.
Haddock, Needell, Rebrova \& Swartworth 
\cite[Proposition 3.4]{hadd} show that for a suitable class of random matrices
$$ \sigma_{q-\beta,\min}(A) \gtrsim (q-\beta)^{3/2} \sqrt{\frac{m}{n}} \qquad \mbox{with high likelihood}$$
which shows that it is comparable to $\sigma_{\max}(A)$ up to constants depending on $q,\beta$. 
Let us now specialize to the case where $A \in \mathbb{R}^{m \times n}$ has each row sampled uniformly at random from the surface measure of $\mathbb{S}^{n-1}$ and suppose that the matrix is large, $m, n \gg 1$, and that the ratio $m/n$ is large. 
Trying to find a subset $S \subset \left\{1,2,\dots, m\right\}$ such that $A_S$ has a small singular value might be difficult, however, we can turn the question around: for a given $x \in \mathbb{S}^{n-1}$, how would we choose $S$ to have 
$$ \|A_S x \|^2 = \sum_{i \in S} \left\langle x, a_i \right\rangle^2 \qquad \mbox{as small as possible?}$$
This is easy: we compute $\left\langle x, a_i\right\rangle^2$ for $1 \leq i \leq m$ and pick $S$ to be the set of desired size corresponding to the smallest of these numbers. Using rotational invariance of Gaussian vectors, we can suppose that $x = (1,0,\dots, 0)$. Then we expect, in high dimensions, that
$$ \left\langle a_i, x\right\rangle \sim \frac{1}{\sqrt{n}} \gamma \qquad \mbox{where}~\gamma \sim \mathcal{N}(0,1).$$
\vspace{-20pt}
\begin{center}
\begin{figure}[h!]
\begin{tikzpicture}[scale=0.6]
\filldraw (2,1) circle (0.05cm);
\shade[ball color=black!50!black] (4,0) arc(0:-180:4 and 1.5) to[out=80,in=100] cycle;
\clip plot[variable=\t,domain=0:360,samples=120] ({1.4+cos(15)*1.7*cos(\t)+sin(15)*0.6*sin(\t)},
{1.4+cos(15)*0.6*sin(\t)-sin(15)*1.7*cos(\t)});
\fill[white] plot[variable=\t,domain=0:360,samples=120] ({1.4+cos(15)*1.7*cos(\t)+sin(15)*0.6*sin(\t)},
{1.4+cos(15)*0.6*sin(\t)-sin(15)*1.7*cos(\t)});
\filldraw (1.5,1.4) circle (0.07cm);
\node at (1.9, 1.4) {$x$};
\end{tikzpicture}
\caption{Removing a small spherical cap around the vector $x$.}
\label{fig:cap}
\end{figure}
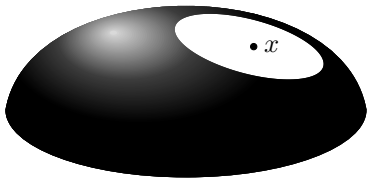
\end{center}

This suggest a certain picture: large inner products are those where many rows $a_i$ are nicely aligned with $x$ and we know with which likelihood to expect them (these are just all the points in the two spherical caps centered at $x$ and $-x$). This would then suggest that, in the limit as $m,n,m/n \rightarrow \infty$, we have
$$ \qquad \qquad \qquad \frac{\sigma_{q-\beta,\min}^2(A)}{ \sigma_{\max}^2(A)} = \frac{1}{\sqrt{2\pi}} \int_{-\alpha}^{\alpha} e^{-x^2/2} x^2 ~dx \qquad \qquad (\diamond)$$
where the parameter $\alpha$ is implicitly defined via
$$  \frac{1}{\sqrt{2\pi}} \int_{-\alpha}^{\alpha}e^{-x^2/2}  dx = q-\beta.$$
It would be interesting to understand whether this, if true, could be rigorously proven. Another interesting question is whether such an asymptotic, if true, could be made quantitative in terms of $m,n, m/n$. A third interesting question is whether for all matrices with $\|a_i\|=1$ there is an inequality of the type
$$ \sigma_{q-\beta,\min}(A)  \leq (1+ o(1)) \cdot c_{q-\beta} \cdot \sqrt{\frac{m}{n}}$$
and whether the best constant $c_{q- \beta}$ is given by the heuristic $(\diamond)$. \\

Independently of these questions, we have the following Corollary. 
\begin{corollary} Let $A \in \mathbb{R}^{m \times n}$ have rows sampled independently and uniformly from $\mathbb{S}^{n-1}$ and $m,n,m/n  \gg 1$ and $0<q<1$. There exists $\delta > 0$ (depending only on $q$) such that the $q-$quantile Random Kaczmarz method converges for $\beta \leq \delta$ (with high probability with respect to $A$).
Assuming $(\diamond)$ and $q =0.88$, we have $\delta \geq 0.0056$.
\end{corollary}

 We emphasize that that this estimate hinges on $(\diamond)$ being correct; however, if $(\diamond)$ were incorrect, then the argument could also be run with another scaling to produce another number.  The fact that there exists such a positive number $\delta>0$ follows quickly from the Main Theorem and \cite[Proposition 3.4]{hadd}. 
 We believe that this Question is interesting in its own right.
\begin{quote}
\textbf{Question 1.} What is the maximum percentage of corruption that the quantile Random Kaczmarz method can absorb for tall Gaussian Random matrices in the asymptotic limit $m,n,m/n \rightarrow \infty$? Is it $(50 - \varepsilon)\%$ or is there a barrier before that?
\end{quote}
The difficulty comes from the fact that the $\beta-$perturbations can be completely arbitrary and adversarial. Maybe there are particular types of corruption effectively exploiting certain idiosyncrasies of the Random Kaczmarz method?  There is a second side to the story: if we assume that the perturbation itself is random (in the sense that the error is supported on $\beta m$ entries but that these entries, interpreted as a vector in $\mathbb{R}^{\beta m}$ are, say, a rescaled Gaussian), then one would naturally expect the $q-$quantile Random Kaczmarz method to be even more effective (since the error cannot effectively conspire against the method).
\begin{quote}
\textbf{Question 2.} What is the maximum percentage of \textit{random} corruption that the quantile Random Kaczmarz method can absorb for Gaussian Random matrices? Can one break the $50\%$ barrier?
\end{quote}
We remark that the numerical evidence in \cite{hadd} (obtained using random perturbations) indicates that the method is actually remarkably stable for such random perturbations even for relatively large amounts of error, even, say, $\beta = 0.5$. It would be tempting to hope that the random case actually allows for $(100-\varepsilon)\%$ corruption (provided the matrix is sufficiently large depending on $\varepsilon$): perhaps $\varepsilon \%$ of consistent structure is actually sufficient to recover the ground truth against $(100-\varepsilon)\%$ of unstructured noise? 

\section{Proof}

\subsection{Outline.} The underlying idea is as follows: we try to bound
$$ \mathbb{E}~ \|x_{k+1} - x\|^2 \qquad \mbox{in terms of} \qquad \| x_k - x\|^2.$$
There are only two cases: when computing $x_{k+1}$ from $x_k$ we either pick an uncorrupted equation or we pick a corrupted equation. When picking an uncorrupted equation, we are in
the classical setting studied by Strohmer \& Vershynin \cite{strohmer} and their argument applies. If we pick a corrupted equation,
then $\|x_{k+1} - x\|^2$ may be larger than $\| x_k - x\|^2$ and our goal is to show that it is not much larger. Finally, we will argue that for a suitable choice of parameters, the expected increase
in size is dominated by the expected decrease coming and this will conclude the result.
The proof decouples into several steps.
\begin{itemize}
\item \S 3.2 gives a bound on the $q-$quantile.
\item \S 3.3 proves Lemma 2, an upper bound on $ \mathbb{E}~\|x_{k+1} - x\|^2$ in terms of $\|x_k - x\|^2$ conditional on having picked a corrupted equation from the $q-$quantile.
\item \S 3.4 rephrases the Strohmer-Vershynin bound for our setting.
\item \S 3.5 combines all ingredients to prove an upper bound on $\mathbb{E} \|x_{k+1} - x\|^2$. 
\item We conclude with several comments and remarks in \S 4.
\end{itemize}

\subsection{A quantile bound} Assume $A \in \mathbb{R}^{m \times n}$ and the underlying equation is 
$$ Ax = b_t,$$
where the subscript $t$ is meant to indicate the true right hand side. We only have access to $b \in \mathbb{R}^m$ which
differs from $b_t$ in at most $\beta m$ entries.
$ \q (y_i)_{i=1}^{m}$
will denote the $q-$th quantile of $m$ real numbers $y_1, \dots, y_m$. 

\begin{lemma}
Let $0 < q < 1-\beta$, let $x_k \in \mathbb{R}^n$ be arbitrary, let $Ax = b_{t}$ and suppose $\| b - b_t\|_{\ell^0} \leq \beta m$. Then
$$ \q\left(  \left| \left\langle x_k, a_i \right\rangle - b_i \right|  \right)_{i=1}^{m} \leq \frac{\sigma_{\max}}{\sqrt{m} \sqrt{1 - q - \beta}} \|x_k - x\|.$$
\end{lemma}
\begin{proof} We index the corrupted equations using $C \subset \left\{1,2,\dots,m\right\}$. Note that $|C| \leq \beta m$ by assumption. We will now consider
the set of uncorrupted equations and note that they satisfy $A_{\notin C} x = b_{\notin C}$ (since $b_{\notin C} = (b_t)_{\notin C}$). 
We start with
$$
  \sum_{i=1 \atop i \notin C}^{m} (\left\langle a_i, x_k \right\rangle- b_i)^2  = \|A_{\notin C} x_k - b_{\notin C}\|^2 = \|A_{\notin C} x_k -  A_{\notin C} x\|^2.
 $$
 This can be bounded from above by
\begin{align*}
 \|A_{\notin C} x_k -  A_{\notin C} x\|^2  &\leq \|A_{\notin C}\|^2\cdot\| x_k -x\|^2  \\
 &\leq \|A_{}\|^2\cdot\| x_k -x\|^2 = \sigma_{\max}^2 \|x_k - x\|^2.
 \end{align*}
Suppose now that 
$$\alpha =  \q\left(  \left| \left\langle x_k, a_i \right\rangle - b_i \right|  \right)_{i=1}^{m} \qquad \mbox{is large}.$$
Then at least $(1-q)m$ of the $m$ numbers $\left| \left\langle x_k, a_i \right\rangle - b_i \right|_{i=1}^{m}$ are at least $\alpha$ and at least $(1-q)m - \beta m$  belong to equations that have not been corrupted. Then
$$ m \left(1 - q - \beta\right) \alpha^2 \leq   \sum_{i \notin C}^{} (\left\langle a_i, x_k \right\rangle- b_i)^2 \leq \sigma_{\max}^2 \|x_k - x\|^2.$$
and therefore
$$ \q\left(  \left| \left\langle x_k, a_i \right\rangle - b_i \right|  \right)_{i=1}^{m} \leq \frac{\sigma_{\max}}{\sqrt{m} \sqrt{1 - q - \beta}} \|x_k - x\|.$$
\end{proof}
\textbf{Remark.} We note that the estimate
$ \| A_{\notin C} \| \leq \|A\|$
is clearly lossy. It would be interesting whether this could be further exploited to get improved estimates.  

\subsection{Bounding Corrupted Equations}
The next step is to provide an upper bound on $\|x_{k+1} - x\|^2$ assuming that we select a corrupted equation. Let us fix $x_k$ and introduce the subset 
$ S \subset C$ of corrupted equations that are simultaneously `almost-correct' equations in the sense of being in the $q-$quantile
$$ S = \left\{i \in C:  \left| \left\langle x_k, a_i \right\rangle - b_i \right|  \leq \q\left(  \left| \left\langle x_k, a_i \right\rangle - b_i \right|  \right)_{i=1}^{m} \right\}.$$
$S$ is the set of all corrupted equations which end up in the set of $q m $ equations that are being considered. If $S = \emptyset$, then when computing $x_{k+1}$ from $x_k$, we have to use an uncorrupted equation and can go straight to \S 3.4. We can thus assume without loss of generality that $|S| \geq 1$. Naturally, we also have $|S| \leq |C| \leq \beta m$. If $i \in \left\{1,2,\dots,m\right\} \setminus S$, then $\|x_{k+1} - x\| \leq \|x_{k} - x\|$, this is simply the standard geometry underlying the Kaczmarz method (or: the Pythagorean theorem). If $i \in S$, then this is no longer true, the distance to the true solution may actually increase. The purpose of this section is to prove that it does not increase too much.
We abbreviate, also throughout the rest of the paper,
$$ \mathbb{E}_{j \in X} Z = \frac{1}{|X|} \sum_{j \in X} Z(j).$$
\begin{lemma} We have
  $$ \mathbb{E}_{i \in S}  \left\| x_{k+1} - x \right\|^2 \leq \left(1 +     \frac{\sigma_{\max}^2}{\sqrt{|S|} \sqrt{m}} \left(\frac{2}{\sqrt{1 - q - \beta}} + \frac{\sqrt{\beta}}{1 - q - \beta} \right)
\right) \|x_k - x\|^2.$$
 \end{lemma}
\begin{proof}
We have, recalling the normalization $\|a_i\|=1$,
$$ x_{k+1} = x_k +  (b_i - \left\langle x_k, a_i \right\rangle  )  a_i.$$
 For any arbitrary vector $v \in \mathbb{R}^n$
$$ \left\| x_k + v - x \right\|^2 =  \left\| x_k - x \right\|^2 + 2 \left\langle x_k - x, v \right\rangle + \|v\|^2$$
and we will apply this to the special choice
$$ v = (b_i - \left\langle x_k, a_i \right\rangle )  a_i \qquad \mbox{where}~i \in S.$$
We first observe that, for $i \in S$, the term $\|v\|^2$ is uniformly small since
\begin{align*}
\|v\|^2 &= \left\|  (b_i - \left\langle x_k, a_i \right\rangle  )  a_i  \right\|^2 = | b_i - \left\langle x_k, a_i \right\rangle  |^2 \\
&\leq \q\left(  \left| \left\langle x_k, a_i \right\rangle - b_i \right|^2  \right)_{i=1}^{m} \leq \frac{\sigma_{\max}^2}{m \left(1-q - \beta\right)} \|x_k - x\|^2.
\end{align*}
It remains to bound $\mathbb{E}_{i \in S}~ 2 \left\langle x_k, v \right\rangle$.
Using the same inequality with Cauchy-Schwarz
\begin{align*}
\mathbb{E}_{i \in S} ~ 2 \left\langle x_k - x, v \right\rangle &=   \frac{2}{|S|} \sum_{i \in S}  \left\langle x_k - x,(b_i - \left\langle x_k, a_i \right\rangle  )  a_i  \right\rangle \\
&=  \frac{2}{|S|} \sum_{i \in S}  (b_i -  \left\langle x_k, a_i \right\rangle)  \left\langle x_k - x, a_i \right\rangle  \\
&\leq \frac{2}{\sqrt{|S|}} \left( \sum_{i \in S}  (b_i -  \left\langle x_k, a_i \right\rangle )^2  \left\langle x_k - x, a_i \right\rangle^2 \right)^{\frac{1}{2}} \\
&\leq  \frac{2}{\sqrt{|S|}} \frac{\sigma_{\max} \cdot \|x_k - x\|}{\sqrt{m} \sqrt{1 - q - \beta}}  \left( \sum_{i \in S}   \left\langle x_k - x, a_i \right\rangle^2 \right)^{\frac{1}{2}}.
\end{align*}
At this point, we estimate
$$ \sum_{i \in S}   \left\langle x_k - x, a_i \right\rangle^2  \leq \sum_{i = 1}^{m}   \left\langle x_k - x, a_i \right\rangle^2  = \|A(x_k-x)\|^2 \leq \sigma_{\max}^2 \|x_k - x\|^2$$
and hence
 $$ \mathbb{E}_{i \in S}~ 2  \left\langle x_k - x, v \right\rangle \leq \frac{2}{\sqrt{|S|}} \frac{\sigma_{\max}^2\cdot \|x_k - x\|^2}{\sqrt{m} \sqrt{1 - q - \beta}} .$$
 Summing up now shows that
 $$ \mathbb{E}_{i \in S}  \left\| x_{k+1} - x \right\|^2 \leq \left(1 + \frac{2}{\sqrt{|S|}} \frac{\sigma_{\max}^2}{\sqrt{m} \sqrt{1-q - \beta}} +  \frac{\sigma_{\max}^2}{m \left(1 - q - \beta\right)} \right) \|x_k - x\|^2.$$
The inequality
 $$|S| \leq \beta m \qquad \mbox{implies} \qquad \frac{1}{m} \leq \frac{\sqrt{\beta}}{\sqrt{m} \sqrt{|S|}} $$
 which leads to
  $$ \mathbb{E}_{i \in S}  \left\| x_{k+1} - x \right\|^2 \leq \left(1 +     \frac{\sigma_{\max}^2}{\sqrt{|S|} \sqrt{m}} \left(\frac{2}{\sqrt{1 - q - \beta}} + \frac{\sqrt{\beta}}{1 - q - \beta} \right)
\right) \|x_k - x\|^2.$$
\end{proof}

\textbf{Remark.} We note that the estimate
$$ \sum_{i \in S}   \left\langle x_k - x, a_i \right\rangle^2  \leq \sum_{i = 1}^{m}   \left\langle x_k - x, a_i \right\rangle^2  = \|A(x_k-x)\|^2 \leq \sigma_{\max}^2 \|x_k - x\|^2$$
could be improved. Clearly, $|S| \leq \beta m$ and thus we could define, analogously to $\sigma_{q-\beta, \min}$, the quantity
$$ \sigma_{\beta,\max}(A) = \max_{S \subset \left\{1,2,\dots, m\right\} \atop |S| =  \beta m } \sup_{x \neq 0} \frac{\|A_S x\|}{\|x\|}$$
and argue that
$$ \sum_{i \in S}   \left\langle x_k - x, a_i \right\rangle^2 \leq \sigma_{\beta, \max}(A)^2 \cdot \|x_k - x\|$$
which would lead to a slight improvement at the cost of introducing an additional quantity, $\sigma_{\beta, \max}$. It is not clear whether this could be reasonably exploited later on since it would require additional estimates on $\sigma_{\beta, \max}$. Though such estimates may be quite doable for, say, random matrices, where one might perhaps expect an estimate along the lines of
$$ \sigma_{\beta, \max}(A) \leq h(\beta) \sqrt{ \frac{m}{n}}$$
and where $h(\beta) \rightarrow 0$ as $\beta \rightarrow 0$. Such an estimate could conceivably be useful for both Question 1 and Question 2 stated in \S 2.3.

\subsection{Bounding uncorrupted equations}
Let $x_k$ be fixed and consider the set of admissible equations
$$ B =  \left\{1 \leq i \leq m:  \left| \left\langle x_k, a_i \right\rangle - b_i \right|  \leq \q\left(  \left| \left\langle x_k, a_i \right\rangle - b_i \right|  \right)_{i=1}^{m} \right\}.$$
\S 3.3 dealt with the subset $S \subset B$ of corrupted equations. Here, we now consider the subset $B \setminus S$ which means applying the Random Kaczmarz method to an uncorrupted equation. The relevant argument is not new and is from \cite{strohmer}.
\begin{lemma}[Strohmer-Vershynin, \cite{strohmer}] We have
$$ \mathbb{E}_{i \in B \setminus S} ~\|x_{k+1} - x\|^2 \leq \left( 1 - \frac{\sigma_{q-\beta,\min}^2}{q m} \right) \|x_k - x\|^2$$
\end{lemma}
\begin{proof} Since we are dealing with respect to uncorrupted equations, one step of the
$q-$quantile Random Kaczmarz method is merely one step of Random Kaczmarz applied to the submatrix $A_{B \setminus S}$. Therefore, using the Strohmer-Vershynin bound,
$$ \mathbb{E}_{i \in B \setminus S} ~\|x_{k+1} - x\|^2 \leq \left( 1 - \frac{\sigma_{\min}(A_{B \setminus S})^2}{\|A_{B \setminus S}\|_F^2} \right) \|x_k - x\|^2.$$
We have, by definition,
$$ \sigma_{\min}(A_{B \setminus S})^2 \geq \sigma_{\frac{|B \setminus S|}{m}, \min}^2$$
and
$$ |B \setminus S| \geq (q - \beta)m$$
from which we get
$$ \sigma_{\min}(A_{B \setminus S})^2 \geq \sigma_{q-\beta, \min}^2.$$
The normalization $\|a_i\|=1$ implies that
$$ \|A_{B \setminus S}\|_F^2 = |B \setminus S| \leq |B| \leq qm.$$
\end{proof}

\subsection{Conclusion} We can now conclude the argument. 

\begin{proof}[Proof of the Theorem]Suppose we are given $x_k$ and the admissible set of equations
$$ B =  \left\{1 \leq i \leq m:  \left| \left\langle x_k, a_i \right\rangle - b_i \right|  \leq \q\left(  \left| \left\langle x_k, a_i \right\rangle - b_i \right|  \right)_{i=1}^{m} \right\}.$$
We recall that the set $S \subset B$ indexes the corrupted equations in the $q-$th quantile.
We have
 \begin{align*}
 \mathbb{E} \|x_{k+1} - x\|^2 &= \mathbb{E}_{i \in S} \|x_{k+1} - x\|^2 \cdot \mathbb{P}\left( i \in S \right) +  \mathbb{E}_{i \in B \setminus S} \|x_{k+1} - x\|^2 \cdot \mathbb{P}\left( i \in B \setminus S \right) \\
 &\leq \mathbb{E}_{i \in S} \|x_{k+1} - x\|^2 \cdot \frac{|S|}{ qm} +  \mathbb{E}_{i \in B \setminus S} \|x_{k+1} - x\|^2 \cdot \left( 1 - \frac{|S|}{q m}\right).
 \end{align*}
 Using Lemma 2
   $$ \mathbb{E}_{i \in S}  \left\| x_{k+1} - x \right\|^2 \leq \left(1 +     \frac{\sigma_{\max}^2}{\sqrt{|S|} \sqrt{m}} \left(\frac{2}{\sqrt{1 - q - \beta}} + \frac{\sqrt{\beta}}{1 - q - \beta} \right)
\right) \|x_k - x\|^2$$
 and Lemma 3
 $$ \mathbb{E}_{i \in B \setminus S} ~\|x_{k+1} - x\|^2 \leq \left( 1 - \frac{\sigma_{q-\beta,\min}^2}{q m} \right) \|x_k - x\|^2,$$
we arrive at
 \begin{align*}
 \mathbb{E} \|x_{k+1} - x\|^2 &\leq \frac{|S|}{qm}\left(1 +     \frac{\sigma_{\max}^2}{\sqrt{|S|} \sqrt{m}} \left(\frac{2}{\sqrt{1 - q - \beta}} + \frac{\sqrt{\beta}}{1 - q - \beta} \right)
\right) \|x_k - x\|^2 \\
&+ \left(1-\frac{|S|}{qm}\right)  \left( 1 - \frac{\sigma_{q-\beta,\min}^2}{q m} \right) \|x_k - x\|^2.
\end{align*}
The upper bound is monotonically increasing in $|S|$, the worst case is $|S| = \beta m$. 
Hence
\begin{align*}
\mathbb{E} \|x_{k+1} - x\|^2 &\leq \frac{\beta }{q}  \left[ 1 +     \frac{\sigma_{\max}^2}{\sqrt{\beta} m} \left(\frac{2}{\sqrt{1 - q - \beta}} + \frac{\sqrt{\beta}}{1 - q - \beta} \right)\right] \|x_k - x\|^2. \\
&+ \left(1 - \frac{\beta }{q}\right)  \left( 1 - \frac{\sigma_{q-\beta, \min}^2}{qm } \right)\| x_k - x\|^2.
\end{align*}

We first rewrite this as
\begin{align*}
\frac{\mathbb{E} \|x_{k+1} - x\|^2}{\|x_k - x\|^2} &\leq  1 +   \frac{  \sigma_{\max}^2}{q m} \left(\frac{2\sqrt{\beta} }{\sqrt{1 - q - \beta}} + \frac{\beta}{1 - q - \beta} \right) -(q-\beta)\frac{ \sigma_{q-\beta,\min}^2}{q^2 m} .
\end{align*}
In order to ensure decay in expectation, we require
$$  \frac{q}{q-\beta}\left( \frac{2\sqrt{\beta} }{\sqrt{1 - q - \beta}} + \frac{\beta}{1 - q - \beta} \right) <  \frac{\sigma_{q-\beta,\min}^2}{ \sigma_{\max}^2}.$$
\end{proof}

\section{Remarks}

\subsection{Computational aspects.} The way we introduce the algorithm, computing the $q-$quantile of $ \left\{ \left| \left\langle x_k, a_i \right\rangle - b_i \right| : 1 \leq i \leq m \right\}$, requires the computation of $m$ different inner products at each step which is computationally expensive.
The algorithm proposed by Haddock, Needell, Rebrova \& Swartworth \cite{hadd} has an additional parameter: pick a certain number $t$ of equations uniformly at random and compute the $q-$quantile with respect to those. $t$ random samples being used to estimate the $q-$quantile reduces computational cost by a factor of $t/m$. However, the main result of \cite{hadd} requires $t=m$: the quantile is computed exactly and the underlying method reduces to the method we described.
Estimating the $q-$quantile using random samples is a rather stable process. In particular, the likelihood of, say, using $t=50$ samples to estimate the median and ending up getting a value in the $99-$th percentile is extremely unlikely $(1-0.99)^{50} = 10^{-100}$ and it is clear that sampling will produce a valuable speed-up in a reliable way. Another difference is that the algorithm does not specify $\|a_i\|=1$, however, the matrices are assumed to belong to certain families of random matrices for which one expects tight concentration of the norm of a row. This assumption of rows being roughly comparable built into the structure of the algorithm (otherwise one would weigh things differently, \S 4.3).

\subsection{Proof of the Corollary} The purpose of this section is to discuss the case of Gaussian Random Matrices subject to the heuristic ($\diamond$) mentioned above. The condition to be checked is
$$  \frac{q}{q-\beta}\left( \frac{2\sqrt{\beta} }{\sqrt{1 - q - \beta}} + \frac{\beta}{1 - q - \beta} \right) <  \frac{\sigma_{q-\beta,\min}^2}{ \sigma_{\max}^2}.$$
 We have 
 $$ \sigma_{\max}^2 \sim (1 + o(1)) \cdot \frac{m}{n}$$
 and, assuming $(\diamond)$, we expect
 $$ \sigma_{q-\beta, \min}^2 = (1+o(1)) \cdot \alpha_{q -\beta} \cdot \frac{m}{n}.$$
 and thus the relevant question is when 
$$  \frac{q}{q-\beta}\left( \frac{2\sqrt{\beta} }{\sqrt{1 - q - \beta}} + \frac{\beta}{1 - q - \beta} \right) <  \alpha_{\beta -q}.$$
is satisfied. Setting $q = 0.88$, we see with some minor computations that the inequality is satisfied for all
$ \beta < 0.0056.$
There is reason to believe (as indicated in various parts of the proof) that the true value is quite a bit larger.

\subsection{Matrices without normalization.} A natural question is whether it is possible to extend these types of considerations to matrices that do not have normalized rows. It is clear that, in such a case, the notion of quantiles will have to be adapted: consider, for example, a matrix $A \in \mathbb{R}^{m \times n}$ where $\|a_1\| \gg \varepsilon^{-1}$ while $\|a_i\| \ll \varepsilon$ for all $2 \leq i \leq m$. As $\varepsilon \rightarrow 0$, we see that the first row is much more important than the other rows and assumes a dominant role. Then, however,  the importance of that particular equation needs to be accounted for in the overall regime. 
In the uncorrupted regime, this is naturally accounted for by the Strohmer-Vershynin scaling: selecting the $i-$th equation with likelihood proportional to $\|a_i\|_{\ell^2}^2$. This suggests changing the definition of corruption away from the size of the support $\|b - b_t\|_{\ell^0} \leq \beta m$ to a condition of the type
$$ \sum_{i \in C} \|a_i\|^2 \leq \beta \sum_{i=1}^{m} \|a_i\|^2 = \beta \cdot \|A\|_F^2.$$
Given the number of open questions even under the assumption $\|a_i\|_{\ell^2}=1$, we have not pursued this alternative but consider it to be very interesting.

\subsection{Different Selection Probabilities} The entire approach in this paper is based on selecting equations with equal likelihood (provided $\|a_i \| =1$). 
However, it is well understood that for classical (uncorrupted) Random Kaczmarz, it is advantageous to pick equations that are violated more strongly more frequently 
\cite{agmon, cenker, motzkin}.  Suppose that $\|a_i\|=1$, that $p \geq 0$ and that
$$ \mathbb{P}(\mbox{we choose equation}~i) =  \frac{\left| \left\langle a_i, x_k \right\rangle - b \right|^p}{\|Ax_k - b \|^p_{\ell^p}},$$
then \cite{stein2} shows that for uncorrupted linear systems
$$ \mathbb{E} \left\| x_k - x \right\|_2^2 \leq \left(1 - \inf_{z \neq 0}  \frac{\|A z \|^{p+2}_{\ell^{p+2}}}{\|Az \|^p_{\ell^p}\|z\|^2_{2}}\right)^k \|x_0 - x\|_2^2$$
which is at least the likelihood of the classical Random Kaczmarz method \cite{strohmer} since
$$ \inf_{z \neq 0} \frac{\|A z \|^{p+2}_{\ell^{p+2}}}{\|A z \|^p_{\ell^p}\|z\|^2_{2}} \geq  \frac{1}{\|A\|_F^2  \cdot \|A^{-1}\|^2}$$
with equality if and only if the singular vector $v_n$ corresponding to the smallest singular value of $A$ has the property that $A v_n$ is a constant vector. 
It seems somewhat conceivable that a similar phenomenon is in effect here: the purpose of the $q-$quantile restriction is to ensure that the impact of corrupted
equations is limited, however, by selecting equations that are barely violated, one certainly slows down the convergence rate. This could be an interesting avenue for further research (also with respect to Question 1 and Question 2 in \S 2.3 since such quantities may be easier to analyze for random matrices).

\subsection{Stochastic Gradient Descent.} Finally, we conclude by noting that problems of the type
$$\|Ax - b\|^2  = \sum_{i=1}^{n} \left( \left\langle a_i, x \right\rangle - b_i\right)^2 \rightarrow \min$$
can be, tautologically, be interpreted as
$$ \sum_{i=1}^{n} f_i(x)^2 \rightarrow \min \qquad \mbox{where} \qquad f_i(x) = \left\langle a_i, x \right\rangle - b_i.$$
The Lipschitz constant of $f_i$ is $\|a_i\|_{\ell^2}$ which motivates thinking of a Random Kaczmarz method as a basic form of stochastic gradient descent (see Needell, Srebro \& Ward \cite{need3}). This analogy is also discussed in Haddock, Needell, Rebrova \& Swartworth \cite{hadd} who describe an analogous algorithm for SGD (see also \cite{chi, dekel, kawa, lic, man}). We believe that the setting of quantile-Random Kaczmarz method applied to corrupted linear system may be a useful (because reasonably explicit) model for understanding the effect of manipulating mini-batches in SGD.

\end{document}